\documentclass[oneside]{amsart}

\usepackage[letterpaper,body={13.0cm,21.5cm}, mag=1000]{geometry}
\usepackage{amssymb}
\usepackage{amsthm}
\usepackage{amscd}

\numberwithin{equation}{section}
\theoremstyle{plain}
\newtheorem{cor}[equation]{Corollary}
\newtheorem{lemma}[equation]{Lemma}

\newtheorem{thm}[equation]{Theorem}

\newtheorem*{thma}{Theorem A}
\newtheorem*{thmb}{Theorem B}

\theoremstyle{definition}

\newtheorem{remark}[equation]{Remark}

\newcommand{\dlabel}[1]{\ifmmode \text{\ttfamily \upshape [#1] } \else
{\ttfamily \upshape [#1] }\fi \label{#1}}

\newcommand{\Z}{\operatorname{Z} }

\begin{document}

\title{On finite capable $p$-groups of class $2$ with cyclic commutator subgroups}

\author{Manoj K.~Yadav}

\address{School of Mathematics, Harish-Chandra Research Institute \\
Chhatnag Road, Jhunsi, Allahabad - 211 019, INDIA}

\email{myadav@mri.ernet.in}

\date{\today}

\begin{abstract}
 We study finite capable $p$-groups $G$ of nilpotency class $2$ such that the commutator subgroup $\gamma_2(G)$ of $G$ is cyclic and the center of $G$ is contained in the Frattini subgroup of $G$.
\end{abstract}
\maketitle

\section{Introduction}

A group $G$ is said to be \emph{capable} if there exists some group $H$ such that $G \cong H/\Z(H)$, where $\Z(H)$ denotes the center of $H$. It is really very interesting (as explained in many papers on the topic) to know the structure of capable group. Mathematicians have studied two aspects of capable groups. The first one is to give an upper bound on the index of $\Z(G)$ in a capable group $G$ in the form of some function of the order of $\gamma_2(G)$, where $\gamma_2(G)$ denotes the commutator subgroup of $G$. Recently this was studied 
in \cite{MI01} and in \cite{PS02, PS05}.
For the related old results, please see the references of \cite{PS05}. The second aspect of the problem is to find the structure of a capable group. This was studied in \cite{RB38}, \cite{BFP79}, \cite{HH90}, \cite{HN96}, \cite{AM05} and \cite{AM07}. A classification of $2$-generated finite capable $p$-groups of nilpotency class $2$ was given by M. R. Bacon and L. C. Kappe in \cite{BK03} for an odd prime $p$, and by A. Magidin (using classification of $2$-generated finite $2$-groups of nilpotency class $2$ from \cite{KVS99}) in \cite{AM06} for $p = 2$.

 Notice that the commutator subgroup of a finite $2$-generated $p$-groups of class $2$ is always cyclic. But a finite  $p$-groups of class $2$ with cyclic commutator subgroup may not be $2$-generated as shown by extra-special $p$-groups. So it is interesting (at least for me) to study finite capable $p$-groups of class $2$ with cyclic commutator subgroups. In this short note we study a part of this problem. Our first result is the following:

\begin{thma}
Let $G$ be finite capable $p$-groups of nilpotency class $2$ with cyclic commutator subgroup $\gamma_2(G)$. Then $G/\Z(G)$ is generated by $2$ elements and $|G/\Z(G)| = 
|\gamma_2(G)|^2$.
\end{thma}
 
We would like to remark that 
 Theorem A  can be deduced from a very general theorem of Isaacs' \cite[Theorem C]{MI01}. But the proof of Isaacs' theorem makes use of three pages of rigorous mathematics. Our proof presented in this note is very short as well as simple.

Our next result makes use of \cite[Corollary 4.3]{BK03} and \cite[Theorem 8.1]{AM06}.

\begin{thmb}
 Let $G$ be finite $p$-groups of nilpotency class $2$ such that $\gamma_2(G)$ is cyclic and $\Z(G) \le \Phi(G)$, where $\Phi(G)$ denotes the Frattini subgroup of $G$. 
Then $G$ is capable if and only if it is isomorphic to one of the groups listed in Theorem \ref{thm1} or Theorem \ref{thm2}.
\end{thmb}

\section{Proofs}

The following concept of isoclinism was introduced by P. Hall in \cite{PH40}. Two groups $G$ and $H$ are said to be \emph{isoclinic} if there exist isomorphisms $\phi: G/\Z(G) \to H/\Z(H)$ and $\theta: \gamma_2(G) \to \gamma_2(H)$ such that $\theta ([g_1, g_2]) = [h_1, h_2]$ for all $g_1, g_2 \in G$, where $h_{i} \Z(H) = \phi(g_{i}\Z(G))$ for $i = 1, \;2$. The resulting pair $(\phi, \theta)$ is called an \emph{isoclinism} of $G$ onto $H$.  Notice that isoclinism is an equivalence relation among groups.

 By using the concept of isoclinism of groups and a result of P. Hall \cite{PH40}, we prove the following lemma:

\begin{lemma}\label{lemma1}
 Let $G$ be a finite capable group. Then there is a finite group $H$ such that $G \cong H/\Z(H)$ and $\Z(H) \le \gamma_2(H)$.
\end{lemma}
\begin{proof}
 Since $G$ is capable, there exists some group $H_1$ such that $G \cong H_1/\Z(H_1)$. Consider the isoclinism family of $H_1$. Then it follows from an important result of P. Hall  \cite{PH40} that this family has a group $H$ (say) such that $\Z(H) \le \gamma_2(H)$. Since $H_1$ and $H$ are isoclinic, we have $H/\Z(H) \cong H_1/\Z(H_1) \cong G$. Now, $H/\Z(H) \cong G$ being finite, it follows from a classical theorem of Schur that $\gamma_2(H)$ is finite. This shows that $\Z(H)$ is finite. Thus $H$ is finite, since both $H/\Z(H)$ and $\Z(H)$ are finite. This completes the proof of the lemma. \hfill $\Box$

\end{proof}

M. Isaacs \cite[Lemma 2.1]{MI01} has given an elegant and self-contained proof of the fact that for a finite capable group $G$, there is some finite group $H$ such that $G \cong H/\Z(H)$.
But we do not know of a proof of our Lemma \ref{lemma1} without using Hall's result on isoclinism.

\begin{cor}\label{cor1}
Let $G$ be a finite nilpotent (solvable) capable group. Then there is a finite nilpotent (solvable) group $H$ such that $G \cong H/\Z(H)$ and $\Z(H) \le \gamma_2(H)$.
\end{cor}

In particular, for a finite $p$-group, we have the following result.

\begin{lemma}\label{lemma2}
 Let $G$ be a finite capable $p$-group. Then there is a finite $p$-group $H$ such that $G \cong H/\Z(H)$ and $\Z(H) \le \gamma_2(H)$.
\end{lemma}
\begin{proof}
 Since $G$ is a finite  capable $p$-group, it follows from Corollary \ref{cor1} that there is a finite nilpotent group $H$ such that $G \cong H/\Z(H)$, $\Z(H) \le \gamma_2(H)$. Thus $H/\Z(H)$ is a $p$-group. Let $P$ be the Sylow $p$-subgroup of $H$. Then $H = P \times K$, where $K$ is a Hall 
$p'$-subgroup of $H$. Since $H/\Z(H)$ is a $p$-group, it follows that $K \le \Z(H)$. Thus $[H, H] = [P \times K, P \times K] = [P, P] \le P$. Since  $\Z(H) \le \gamma_2(H)$, we have $K \le P$. Thus $K$ must be trivial. This proves that $H$ is a $p$-group and the proof of the lemma is complete. \hfill $\Box$
 
\end{proof}

We need the following lemma, which is also of an independent interest, for the proof of Theorem A.

\begin{lemma}\label{lemma3}
Let $G$ be a finite $p$-group of class $3$ such that $\gamma_2(G)\Z(G)/\Z(G)$
is cyclic. Then $G/\Z_2(G)$ is generated by $2$ elements.
\end{lemma}
\begin{proof}
Let $\{x_1, x_2, \cdots, x_d\}$ be a minimal generating set for $G$. Since $\gamma_2(G)\Z(G)/$ $\Z(G)$ is
cyclic, we can assume without loss of generality that $[x_1,  x_2]\Z(G)$ generates $\gamma_2(G)\Z(G)/\Z(G)$.  Then $[x_1, x_j] = [x_1, x_2]^{\alpha_j}$ modulo $\Z(G)$ for some integer  $\alpha_j$, where $3 \le j \le d$. Since the nilpotency class of  $G/\Z(G)$ is $2$, it follows that $[x_1, x_jx_2^{-\alpha_j}] \in \Z(G)$. Similary if $[x_2, x_j] = [x_1, x_2]^{\beta_j}$ modulo $\Z(G)$ for some integer  $\beta_j$, where $3 \le j \le d$, then $[x_2, x_1^{\beta_j}x_j] \in \Z(G)$.
Thus it follows that $[x_i, x_1^{\beta_j}x_jx_2^{-\alpha_j}] \in \Z(G)$ for all $i, \; j$ such that $1 \le i \le 2$ and $3 \le j \le d$. Let us set $y_j = x_1^{\beta_j}x_jx_2^{-\alpha_j}$, where $3 \le j \le d$. Notice that the set $\{x_1, x_2, y_3, \cdots, y_d\}$ also generates $G$. 
Thus it follows that $[x_i, y_j] \in \Z(G)$ for all $1 \le i \le 2$ and $3 \le j \le d$. 
Also $[y_j, x_i] \in \Z(G)$ for the same values of $i$ and $j$. Therefore by Hall-Witt identity
\[[x_i, y_j, y_k][y_k, x_i, y_j][y_j, y_k, x_i] = 1\] 
for $1 \le i \le 2$ and $3 \le j, k \le d$, we get 
\begin{equation}\label{equation1}
[y_j, y_k, x_i] = 1.
\end{equation}
Similarly $[x_1, x_2, y_l] = 1$, since $[y_l, x_1, x_2] = 1 = [x_2, y_l, x_1] = 1$, where 
$3 \le l \le d$. Now for all $j, k$ such that $3 \le j, k \le d$, we have
$[y_j, y_k] = [x_1, x_2]^tz$ for some integer $t$ and some $z \in \Z(G)$.
Since $[x_1, x_2, y_l] = 1$, this gives $[y_j, y_k, y_l] = 1$, where $3 \le j,
k, l \le d$. This, along with \eqref{equation1}, proves that 
$[y_j, y_k] \in \Z(G)$ for all $j, k$ such that $3 \le j, k \le d$. 
Since $[y_j, x_i] \in \Z(G)$ for  $3 \le j \le d$ and 
$1 \le i \le 2$, it follows that $[y_j, G] \subseteq \Z(G)$ for all $j$ such that
$3 \le j \le d$. Hence $y_j \in \Z_2(G)$ for $3 \le j \le d$. This proves that $G/\Z_2(G)$ is generated by $x_1\Z_2(G)$ and $x_2\Z_2(G)$, which completes the proof of the lemma.
\hfill $\Box$

\end{proof}

Now we prove Theorem A.

\noindent \emph{Proof of Theorem A.}
 Since $G$ is a finite capable $p$-group of class $2$ such that $\gamma_2(G)$ is cyclic, there exists a $p$-group $H$ of nilpotency class $3$ such that $G \cong H/\Z(H)$, $\Z(H) \le \gamma_2(H)$ and $\gamma_2(H)/\Z(H)$ is cyclic. It now follows from Lemma \ref{lemma3} that $H/\Z_2(H)$ is generated by $2$ elements. This, in turn, implies that $G/Z(G)$ is generated by $2$ elements. Since the nilpotency class of $G$ is $2$,  $G/\Z(G)$ and $\gamma_2(G)$ have the same exponent $p^e$ (say) and cyclic decomposition of $G/\Z(G)$ has at least two cyclic factors of order $p^e$. Since $G/\Z(G)$ is generated by $2$ elements, any cyclic decomposition of $G/\Z(G)$ has two cyclic factors, each of  order $p^e$.  Thus $|G/\Z(G)| = p^{2e} = |\gamma_2(G)|^2$. This completes the proof of the theorem.    \hfill $\Box$

\begin{cor}\label{cor2}
Let $G$ be a finite capable $p$-group of class $2$ such that $\Z(G) \le \Phi(G)$. Then $\gamma_2(G)$ is cyclic if, and only if $G$ is a $2$-generated group.
\end{cor}
\begin{proof}
Since $\gamma_2(G)$ is cyclic for every $2$-generated finite $p$-group $G$ of class $2$,  we only need to prove the if part of the corollary to complete the proof.
Let $\gamma_2(G)$ be cyclic. Since $G$ is finite capable $p$-group of class $2$, it follows from Theorem A that $G/\Z(G)$ is $2$-generated. Now using the hypothesis  $\Z(G) \le \Phi(G)$, one can immediately see that $G/\Phi(G)$ is $2$-generated. This proves that $G$ is $2$-generated. \hfill $\Box$

\end{proof}

We need the following two theorems for the proof of Theorem B. 

\begin{thm}[Theorem 8.1, \cite{AM06}]\label{thm1}
Let $G$ be a $2$-generated finite $2$-group of nilpotency class $2$. Then $G$ is capable if and only if it is isomorphic to one of the following groups:
\begin{eqnarray*} 
\ref{thm1}(i) \hspace{.55in}G &=& \langle a, b \;|\; a^{2^{\alpha}} = b^{2^{\beta}} = [a, b]^{2^{\gamma}} = [a, b, a] = [a, b, b] = 1 \rangle,\\
& & \text{where} \;\; \alpha, \beta, \gamma \;\; \text{are positive integers satisfying} \;\;\alpha = \beta\;\; \text{or} \\
& &\alpha = \beta + 1 = \gamma + 1.\\
\ref{thm1}(ii) \hspace{.50in} G &= &\langle a, b \;|\; a^{2^{\alpha}} = b^{2^{\beta}} = [a, b, a] = [a, b, b] = 1, a^{2^{\alpha + \sigma -\gamma}} = [a, b]^{2^{\sigma}} \rangle,\\
& & \text{where} \;\; \alpha, \beta, \gamma, \sigma \;\; \text{are positive integers satisfying}\;\; \alpha = \beta\\
& &\text{and}\;\; \gamma < \beta - 1\;\; \text{or} \;\; \alpha = \beta + 1 = \gamma + 1 = \sigma + 2.\nonumber
\end{eqnarray*}
\end{thm}

\begin{thm}[Corollary 4.3, \cite{BK03}]\label{thm2}
Let $G$ be a $2$-generated finite $p$-group of nilpotency class $2$, where $p$ is an odd prime. Then $G$ is capable if and only if it is isomorphic to one of the following groups:
\begin{eqnarray*}
\ref{thm2}(i) \hspace{.55in} G &=& \langle a, b \;|\; a^{p^{\alpha}} = b^{p^{\beta}} = [a, b]^{p^{\gamma}} = [a, b, a] = [a, b, b] = 1 \rangle,\\
& & \text{where}\;\; \alpha, \beta, \gamma \;\; \text{are positive integers satisfying}\;\;\alpha = \beta \ge \gamma.\nonumber\\
\ref{thm2}(ii) \hspace{.5in} G &=& \langle a, b \;|\; a^{p^{\alpha}} = b^{p^{\beta}}  = [a, b, a] = [a, b, b] = 1, [a, b] = a^{p^{\alpha - \gamma}}\rangle,\\
& & \text{where}\;\;\alpha, \beta, \gamma \;\; \text{are positive integers satisfying}\;\; \alpha = \beta \ge 2 \gamma.\nonumber
\end{eqnarray*}
\end{thm}

\noindent \emph{Proof of Theorem B.} Using Corollary \ref{cor2}, the proof immediately follows from Theorems \ref{thm1} and \ref{thm2}. \hfill $\Box$
 
\begin{remark}
It may be interesting (but I do not any reason for this) to classify finite capable $p$-groups $G$ of nilpotency class $2$ such that $\gamma_2(G)$ is cyclic. 
\end{remark}

\end{document}